\documentclass{amsart}

\usepackage[english]{babel}
\usepackage[T1]{fontenc}
\usepackage{indentfirst}
\usepackage[utf8]{inputenc}
\usepackage{amsmath}
\usepackage{amsfonts}
\usepackage{amssymb}
\usepackage{mathrsfs}
\usepackage{amsthm}
\usepackage{xfrac}
\usepackage{lmodern}
\usepackage{fix-cm}
\usepackage{listings}
\usepackage{fancyvrb}
\usepackage{enumerate}   

\newcommand{\minh}{\mathcal{M}\mathrm{inh}}
\newcommand{\maxh}{\mathcal{M}\mathrm{axh}}

\newcommand{\cdeg}{\mathrm{cdeg}}
\newcommand{\IsInT}{\mathrm{IsInT}(v,G,\mathrm{var} \ c \texttt{\_}T)}
\newcommand{\IsInTv}{\mathrm{IsInT}(v,G, c \texttt{\_}T)}
\newcommand{\IsInTw}{\mathrm{IsInT}(w,G, c \texttt{\_}T)}
\newcommand{\Tor}{\mathrm{Tor}}

\let\To=\longrightarrow

\def\CC{{\mathcal C}}

\def\opn#1#2{\def#1{\operatorname{#2}}} 
\opn\chara{char} \opn\length{\ell} \opn\pd{pd} \opn\rk{rk}
\opn\projdim{proj\,dim} \opn\injdim{inj\,dim} \opn\rank{rank}
\opn\depth{depth} \opn\grade{grade} \opn\height{height}
\opn\embdim{emb\,dim} \opn\codim{codim}

\opn\Tr{Tr} \opn\bigrank{big\,rank}
\opn\superheight{superheight}\opn\lcm{lcm}
\opn\trdeg{tr\,deg}
\opn\reg{reg} \opn\lreg{lreg} \opn\ini{in} \opn\lpd{lpd}
\opn\size{size}\opn\bigsize{bigsize}
\opn\cosize{cosize}\opn\bigcosize{bigcosize}
\opn\sdepth{sdepth}\opn\sreg{sreg}
\opn\link{link}\opn\fdepth{fdepth}

\newtheoremstyle{theorem}
{12pt} 
{12pt} 
{\slshape 

} 
{} 
{\bfseries} 
{} 
{ } 
{} 

\newtheoremstyle{definition}
{12pt} 
{12pt} 
{\slshape 

} 
{} 
{\bfseries} 
{} 
{ } 
{} 

\newtheoremstyle{algorithm}
{10pt} 
{10pt} 
{\slshape 

} 
{} 
{\bfseries} 
{} 
{ } 
{} 

\theoremstyle{theorem}
\newtheorem{theorem}{Theorem}
\newtheorem{corollary}[theorem]{Corollary}
\newtheorem{proposition}[theorem]{Proposition}
\newtheorem{lemma}[theorem]{Lemma}

\theoremstyle{definition}
\newtheorem{definition}[theorem]{Definition}

\theoremstyle{remark}
\newtheorem{remark}[theorem]{Remark}
\newtheorem{example}[theorem]{Example}

\theoremstyle{algorithm}

\numberwithin{theorem}{section}

\title[Krull dimension and regularity of block graphs]{Krull dimension and regularity of binomial edge ideals of block graphs}
\author{Carla Mascia}
\address{University of Trento}
\email{carla.mascia@unitn.it}
\author{Giancarlo Rinaldo}
\address{University of Trento}
\email{giancarlo.rinaldo@unitn.it}

\begin{document}

\maketitle

\begin{abstract}
We give a lower bound for the Castelnuovo-Mumford regularity of binomial edge ideals of block graphs by computing the two distinguished extremal Betti numbers of  a new family of block graphs, called flower graphs. Moreover, we present linear time algorithms to compute the Castelnuovo-Mumford regularity and the Krull dimension of binomial edge ideals of block graphs. 
\end{abstract}

\section*{Introduction}
 In 2010, binomial edge ideals were introduced in \cite{HHHKR} and also appeared independently in \cite{MO}. Let $S = K[x_i, y_j]_{1 \leq i,j \leq n}$ be the polynomial ring in $2n$ variables with coefficients in a field $K$. Let $G$ be a graph on vertex set $[n]$ and edges $E(G)$. The ideal $J_G$ of $S$ generated by the binomials $f_{ij} = x_iy_j - x_jy_i$ such that $i<j$ and $\{i,j\} \in E(G)$ is called the \textit{binomial edge ideal} of $G$. Any ideal generated  by a set of $2$-minors  of a $2\times n$-matrix of indeterminates may be viewed as the binomial edge ideal of a graph. 
 
For a set $T \subset [n]$, let $G_{[n]\setminus T}$ be the induced subgraph of $G$ with vertex set $[n] \setminus T$ and $G_1, \dots, G_{c(T)}$ the connected components of $G_{[n] \setminus T}$. $T$ is  a cutset of $G$ if $c(T \setminus \{i\}) < c(T)$ for each $i \in T$, and we denote by $\mathcal{C}(G)$ the set of all cutsets for $G$. In \cite{HHHKR} and \cite{MO} the authors gave a description of the primary decomposition of $J_G$ in terms of prime ideals induced by the set $\CC(G)$ (see \eqref{Eq:primarydec}). Thanks to this result the following formula for the Krull dimension is obtained
 \begin{equation}\label{Eq:dimension}
   \dim S/J_G = \max_{T \in\CC(G)} \{n +c(T) - |T|\}.
 \end{equation}
 The second author in \cite{R1} described an algorithm to compute the primary decomposition \eqref{Eq:primarydec}, and hence the Krull dimension.  Unfortunately, this algorithm is exponential in time and space.

A \textit{block graph}, also known as clique tree, is a graph whose blocks are cliques. In general, computing the depth of a ring is a difficult task. In \cite{EHH} the authors prove that when $G$ is a block graph, $\depth S/J_G=n+c$ and, equivalently, $\projdim S/J_G=n-c$, where $c$ is the number of connected components of $G$. Given any block graph, such an immediate formula for the Krull dimension does not exist, and we believe that is hard to obtain something that is better than \eqref{Eq:dimension}.

Nevertheless we believe the block graphs are the easier but not trivial class where we can obtain a good algorithm to compute the Krull dimension. We present an algorithm (Theorem \ref{Alg for Krull dim}) that is linear in time and space and computes the Krull dimension. The idea is to find a minimal prime ideal of minimum height since it induces the Krull dimension of $S/J_G$. We have implemented the algorithm using CoCoA (\cite{Co}), when $G$ is a tree and it is freely downloadable on \cite{MR}.

Another fundamental invariant that has been studied in deep is the Castelnuovo-Mumford regularity of binomial edge ideal. Lower and upper bounds for the regularity are known by  Matsuda and Murai \cite{MM} and Kiani and Saeedi Madani \cite{KM}. In \cite{EZ}, the authors proved the conjecture posed in \cite{MK2} for closed graphs and block graphs. For these graphs, the regularity of $S/J_G$ is bounded below by the
length of the longest induced path of $G$ and above by $c(G)$, where $c(G)$ is the number of maximal cliques of $G$. 
Furthermore, Kiani and Saeedi Madani  characterized all graphs whose binomial edge ideal have regularity 2 and regularity 3, see \cite{MK} and \cite{MK1}.

It is still an open problem to determine an explicit formula for the regularity of binomial edge ideals for block graphs in terms of the combinatorics of the graph. Recently, in \cite{HR} Herzog and the second author computed one of the distinguished extremal Betti number of the  binomial edge ideal of a block graph and classify all block graphs admitting precisely one extremal Betti number giving a natural lower bound for the regularity of any block graph. Jayanthan et al in \cite{JNR} and in \cite{JNR2} obtained a related result for trees, a subclass of block graphs. 

Inspired by these results we define a new class of graphs, namely the flower graphs (see Definition \ref{Def:BigFlower} and Figure \ref{Big Flower}), for which we compute the superextremal Betti numbers (see Theorem \ref{main theo}) and the regularity (see Corollary \ref{cor:reg flower}). As a consequence we obtain new lower bounds in Theorem \ref{The:LowerBound} and Corollary \ref{Cor:LowerBound} for the regularity of any block graph. 

In Section \ref{sec:algreg}, we state the main result of this work, Theorem \ref{Theo: reg block graph}, that provides an efficient method to compute the Castelnuovo-Mumford regularity of any binomial edge ideal of block graphs, exploiting the notion of end-flowers (see Definition \ref{Def: end-flowers}) and by means of an unique block graph traversal.

\section{On the height of minimal prime ideals of $J_G$ and decomposability of block graphs}

We start this section by recalling the formula to compute the primary decomposition of a binomial edge ideal $J_G$. Let $G$ be a graph on $[n]$, $T \subset [n]$ is called \textit{cutset} of $G$ if $c(T \setminus \{i\}) < c(T)$ for each $i \in T$, where $c(T)$ denotes the number of connected components induced by removing $T$ from $G$. We denote by $\mathcal{C}(G)$ the set of all cutsets of $G$. When $T \in \mathcal{C}(G)$ consists of one vertex $v$, $v$ is called a \textit{cutpoint}. Define
$$
P_T(G) = \left( \bigcup_{i \in T} \{x_i, y_i\}, J_{\tilde{G}_1}, \dots, J_{\tilde{G}_{c(T)}} \right) \subseteq S
$$
where $\tilde{G}_i$, for $i=1, \dots, c(T)$, denotes the complete graph on $V(G_i)$.
$P_T(G)$ is a prime ideal of height $n-c(T)+ |T|$. It holds 
\begin{equation}\label{Eq:primarydec}
J_G = \bigcap_{T \in \CC(G)} P_T(G). 
\end{equation}

We denote by $\mathcal{M}(G)$ the minimal prime ideals of $J_G$, by $\minh(G) \subseteq \mathcal{M}(G)$ the minimal prime ideals $P_T(G)$ of minimum height and by $\maxh(G) \subseteq \mathcal{M}(G)$ the minimal prime ideals $P_T(G)$ of maximum height.

A subset $C$ of $V(G)$ is called a \textit{clique} of $G$ if for all $i, j \in C$, with $i \neq j$, one has $\{i,j\} \in E(G)$. A \textit{maximal clique} is a clique that cannot be extended by including one more adjacent vertex.

A connected subgraph of $G$ that has no cutpoint and is maximal with respect to this property is a \textit{block}. $G$ is called \textit{block graph} if all its blocks are complete graphs. One can see that a graph $G$ is a block graph if and only if it is a chordal graph in which every two maximal cliques have at most one vertex in common.  Let $G$ be a block graph, an \textit{endblock} of $G$ is a block having exactly one cutpoint. 

The \textit{clique degree} of $v$, denoted by $\cdeg(v)$, is the number of maximal cliques to which $v$ belongs. A vertex $v$ is called a \textit{free vertex} of $G$ if $\cdeg(v) =1$, and is called an \textit{inner vertex} of $G$ if $\cdeg(v) >1$.

\begin{definition}
A graph $G$ is \textit{decomposable} if exists  a decomposition
\begin{equation}\label{G1 u G2}
G = G_1 \cup G_2
\end{equation}
with $V(G_1) \cap V(G_2) = \{v\}$ such that $v$ is a free vertex of $G_1$ and $G_2$. If $G$ is not decomposable, we call it \textit{indecomposable}. By a recursive decomposition (\ref{G1 u G2}) applied to each $G_1$ and $G_2$, after a finite number of steps we obtain
\begin{equation}\label{dec unique}
G = G_1 \cup \dots \cup G_r 
\end{equation}
where $G_1, \dots, G_r $ are indecomposable and for $1 \leq i < j \leq r$ either $V(G_i)\cap V(G_j) = \emptyset$ or $V(G_i)\cap V(G_j) = \{ v_{ij}\}$, where $v_{ij}$ is a free vertex of $G_i$ and $G_j$. The decomposition (\ref{dec unique})  is unique up to ordering and we say that G is decomposable into indecomposable graphs $G_1, \dots,G_r$.
\end{definition}

For the sake of completeness, we collect in the next proposition the results showed in  \cite{HR} and \cite{RR} concerning Krull dimension of $S/J_G$, height of the ideals $P_T(G)$ and the Castelnuovo-Mumford regularity of $S/J_G$, when $G$ is decomposable.

\begin{proposition}\label{prop:sum of invar}
Let $G$ be a graph decomposable into $G_1$ and $G_2$, with $V(G_1) \cap V(G_2) = \{v\}$. Then
\begin{enumerate}
\item $\dim S/J_G = \dim S_1/J_{G_1} + \dim S_2/J_{G_2}-2$, where $S_i = K[x_j,y_j]_{j \in V(G_i)}$ for $i=1,2$;
\item $\height P_T(G) = \height P_{T_1}(G_1) + \height P_{T_2}(G_2)$,  with $T \in \mathcal{C}(G)$, $T_1 \in \mathcal{C}(G_1)$,  and $T_2 \in \mathcal{C}(G_2)$ and  either $T = T_1 \cup T_2$ or $T = T_1 \cup T_2 \cup \{v\}$;
\item $\reg S/J_G = \reg S/J_{G_1} + \reg S/J_{G_2}$.
\end{enumerate}
\end{proposition}

For a block graph $G$, being decomposable can be read from the primary decomposition of $J_G$, in particular from the ideals in $\maxh(G)$.

\begin{proposition}\label{prop:indec TFAE}
Let $G$ be a block graph. The following are equivalent:
\begin{enumerate}
\item G is indecomposable;
\item if $v \in V(G)$, then $\cdeg(v)\neq 2$;
\item $\maxh(G) = \{P_{\emptyset}(G)\}$.
\end{enumerate}
\end{proposition}

\begin{proof}
\  
\begin{enumerate}
\item[(1)] $\Leftrightarrow \mathrm{(2)}$ It is trivial. 
\item[(2)] $\Rightarrow \mathrm{(3)}$ Without loss of generality,  let $G$ be connected. Since $\height P_{\emptyset}(G) =n-1$, we want to prove that for any $T \neq \emptyset$, $\height P_T(G) < n-1$ . Let $T \in \mathcal{C}(G)$, with $\height P_T(G) \geq n-1$, that is $c(T) - |T| \leq 1$. If $T= \{v\}$, then $c(T) \leq 2$ or equivalently $\cdeg(v) \leq 2$. Since $v$ is a cutpoint, it is not a free vertex, and then $\cdeg(v) =2$, which is in contradiction to the hypothesis. Let $T = \{v_1, \dots, v_r\}$, with $r \geq 2$, such that $\height P_T(G) \geq n-1$ and suppose it is  minimal with respect to this property. In a block graph, $T_1 = T\setminus \{v_r\}$ is a cutset, too. By definition, $c(T_1) < c(T)$ and $|T_1|= |T|-1$, then $c(T_1)-|T_1| <2$. It follows that $\height P_{T_1}(G) \geq n-1$, but it is in contradiction to the hypothesis on the minimality of $T$.
\item[(3)] $\Rightarrow \mathrm{(2)}$ Assume that there exists a vertex $v \in V(G)$ such that $\cdeg(v)=2$. Let $T=\{v\}$, then $\height P_T(G) = \height P_{\emptyset}(G)= n-1$. Hence, $P_T(G) \in \maxh(G)$, too. The latter is in contradiction to the hypothesis.  
\end{enumerate}

\end{proof}

We observe that for a generic graph $G$, is not true that if $G$ is indecomposable then $\cdeg(v)\neq 2$ for any $v \in V(G)$. It is sufficient to consider $G=C_4$, with $V(G)=\{1, \dots,4\}$ and $E(G) = \{\{i,i+1\} | i=1,\dots,3\} \cup \{1,4\}$. All its vertices have clique degree equal to 2, but $G$ is indecomposable. Moreover, for a generic graph $G$ being indecomposable is not equivalent to the fact that $P_{\emptyset}(G)$ is the prime ideal  of the maximum height in the primary decomposition of $J_G$. In fact, consider again $G = C_4$. The subset $T=\{1,3\}$ is a cutset for $G$ and $\height P_T(G) = 4$, whereas $\height P_{\emptyset}(G) = 3$.

\section{Krull dimension of binomial edge ideals of block graphs}

If $G$ is any graph with $n$ vertices, the Krull dimension of $S/J_G$ is given by $\dim S/J_G  = \max_{T \in \mathcal{C}(G)} \{n + c(T) -|T|\}$, and then, in general, to compute it one must investigate all the possible cutsets of $G$. For some classes of graphs, there exists an immediate way to compute the Krull dimension. For example, if $G$ is a complete graph or a graph obtained by gluing free vertices of complete graphs and such that any vertex $v \in V(G)$ is either a free vertex or has $\cdeg(v) =2$, then $\dim S/J_G =n+1$. For a generic block graph $G$, we show an algorithm to compute the Krull dimension of $S/J_G$ in linear time.

From now on, we consider only connected block graphs, since the Krull dimension of $S/J_G$, where $G$ is a graph with $c$ connected components, $G_1, \dots, G_c$, is given by the sum of the Krull dimension of $S_i/J_{G_i}$, with $i=1, \dots, c$ and $S_i = K[x_j,y_j]_{j \in V(G_i)}$. Before showing the aforementioned algorithm, we need some auxiliary results.

\begin{lemma}\label{v with at least 2 leaves}
Let $G$ be a block graph,  $P_T(G) \in \minh(G)$, and $v \in V(G)$. If $v$ belongs to 
\begin{enumerate}
\item exactly two endblocks, then $P_{T\cup\{v\}}(G) \in \minh(G)$; 
\item at least three endblocks, then $v \in T$.
\end{enumerate}
\end{lemma}

\begin{proof}
Let $P_T(G) \in \minh(G)$ and let $v$ belong to $r$ endblocks, $B_1, \dots, B_r$, with $r \geq 2$, and let $G_1, \dots, G_c$ be the connected components of  $G_{[n] \setminus T}$, then $\mathrm{height } P_T(G) = n - c+ |T|$. Suppose that $v \not \in T$. Without loss of generality, we can suppose $v \in G_1$. The connected components induced by $T\cup\{v\}$ are $B'_1,\dots,B'_r, G'_1, G_2, \dots, G_c$, where $B'_i = B_i \setminus \{v\}$ for $i=1,\dots,r$ and $G'_1 = G_1 \setminus \{B_1, \dots, B_r\}$. If $r=2$ and $G'_1 = \emptyset$, then $\mathrm{height } P_{T\cup\{v\}}(G) = \mathrm{height } P_T (G)$, and then also $P_{T\cup\{v\}}(G) \in \minh(G)$. If $r \geq 3$ or $r=2$ and $G'_1 \neq \emptyset$, the number of connected components induced by $T\cup\{v\}$ is at least $r+c$ and hence it is greater than or equal to $c+2$. Thus, $\mathrm{height } P_{T\cup\{v\}}(G) \leq n -(c+2) + (|T|+1) < \mathrm{height } P_T (G)$, which is in contradiction to the minimality of $P_T(G)$. 
\end{proof}

\begin{remark}\label{T not vert of degree 2}
Let $G$ be a block graph and $T \in \mathcal{C}(G)$ such that $P_T(G) \in \minh(G)$. If  $\{v_1, \dots, v_r\} \subseteq T$ is the set of all the vertices in $T$ with clique degree equal to 2, by Proposition \ref{prop:sum of invar}.(2), $P_{T \setminus \{v_1, \dots, v_r\}}(G) \in \minh(G)$.
\end{remark}

\begin{lemma}\label{Lemma:v union min = min}
Let $G$ be a block graph and $v \in V(G)$ be a cutpoint. If 
\begin{enumerate}
\item $v$ belongs to at least 2 endblocks of an indecomposable component of $G$,
\item $P_{T'}(H) \in \minh(H)$, where $T' \in \CC(H)$ and $H$ is the graph obtained from $G$ by removing $v$ and the endblocks to which $v$ belongs
\end{enumerate}
then $P_{T' \cup \{v\}}(G) \in \minh(G)$.
\end{lemma}

\begin{proof}
Let $T \in \mathcal{C}(G)$ be such that $P_T(G) \in \minh(G)$ and $v \in T$. By Lemma \ref{v with at least 2 leaves}, we know that such $T$ exists. Let $T= T_1 \cup \{v\}$. Let $r \geq 2$ be the number of endblocks to which $v$ belongs, then $c(T) = r +c(T_1)$, where $c(T_1)$ denotes the number of connected components of $H$ induced by $T_1$. It follows that
\begin{align*}
\height P_T(G) &= n-(r+c(T_1))+(1+|T_1|) \\
&= n-V(H)-r+1 + [V(H) -c(T_1)+ |T_1|] \\
&= s + \height P_{T_1}(H)
\end{align*}
where $s = n-V(H)-r+1$. Observe that $P_{T_1}(H) \in \minh(H)$: if there exists $T_2 \in \CC(G)$ such that $\height P_{T_2}(H) < \height P_{T_1}(H)$, then $\height P_{T_2 \cup \{v\}}(G)$ is lower than $\height P_T(G)$, and this is in contradiction to the minimality of $P_T(G)$. Since, by hypothesis, $P_{T'}(H), P_{T_1}(H)$ have the same height, it follows $ \height P_T(G)=  s+  \height P_{T'}(H) = \height P_{T' \cup \{v\}}(G)$, and $P_{T' \cup \{v\}}(G) \in \minh(G)$.
\end{proof}

The following result is the core of the algorithm that allows to compute the Krull dimension of $S/J_G$. 

\begin{theorem}\label{Minheight}
Let $G$ be a block graph and $T=\{v_1,\ldots,v_t\}\in \CC(G)$. We denote by $H_0$ the graph $G$ and by $H_i$ the graph obtained from $H_{i-1}$ by removing $v_i$ and the endblocks to which $v_i$ belongs,  for all $i=1, \dots, t$. If  
\begin{enumerate}
 \item $v_i$ belongs to at least $2$ endblocks of an indecomposable component of $H_{i-1}$, for all $i=1,\ldots,t$,
 \item $H_t$ is decomposable into blocks,
\end{enumerate}
then $P_T(G)\in \minh(G)$.
\end{theorem}

\begin{proof}
We use induction on $t$. Let $t=1$. Consider $T=\{v_1\} \in \mathcal{C}(G)$, with $v_1 \in V(G)$ that belongs to at least $2$ endblocks of an indecomposable component of $G$ and $H_1$ is decomposable into blocks. By Lemma \ref{v with at least 2 leaves} and Remark \ref{T not vert of degree 2}, there exists a cutset $T'$ that contains $v_1$ and no vertices of clique degree equal to 2 such that $P_{T'}(G) \in \minh(G)$.  Since $H_1$ is decomposable into blocks, all the non-free vertices of $H_1$ have clique degree equal to 2, then $T' = \{v_1\} =T$ and $P_T(G) \in \minh(G)$.

\noindent Let $t >1$.  Consider $T=\{v_1,\ldots,v_t\}\in \CC(G)$. The vertex $v_1$ belongs to at least 2 endblocks of an indecomposable component of $G$ and, by induction hypothesis, $P_{T'}(H_1) \in \minh(H_1)$, where $T' = \{v_2, \dots, v_t\}$. By Lemma \ref{Lemma:v union min = min}, $P_{T' \cup \{v_1\}}(G) \in \minh(G)$. 
\end{proof}

\begin{theorem}[Algorithm: Krull Dimension of binomial edge ideals of block graphs]  \label{Alg for Krull dim}
\ 
\begin{itemize}
\item Input: A connected block graph $G$ over $[n]$.
\item Output: Krull dimension of $S/J_G$.
\end{itemize}
\begin{enumerate}[{\footnotesize \hspace{0.08cm} 1.}]
\item $\dim := n+1$; 
\item $\mathcal{G} := \{G\}$;
\item for every graph $H \in \mathcal{G}$
\item \hspace{0.4cm} $\mathcal{G} := \mathcal{G} \setminus \{H\}$; 
\item \hspace{0.4cm} decompose $H$ into its indecomposable subgraphs $\mathcal{I}=\{G_1,\dots,G_r\}$; 
\item \hspace{0.4cm} remove from $\mathcal{I}$ the graphs which are blocks; 
\item \hspace{0.4cm} for every graph  $G_i \in \mathcal{I}$ 
\item \hspace{0.8cm} take $v \in V(G_i)$ such that $v$ belongs to at least 2 endblocks; 
 \item \hspace{0.8cm} $\dim := \dim + \cdeg(v) - 2$;
\item \hspace{0.8cm} $\mathcal{G} := \mathcal{G} \cup \{H_v\}$;
\end{enumerate}
where $H_v$ denotes the graph obtained from $G_i$ by removing $v$ and the endblocks to which $v$ belongs. 
\end{theorem}

\begin{proof}
The aim of the algorithm is to compute the Krull dimension by finding a cutset $T$ such that $P_T(G) \in \minh(G)$. In particular, after a finite number of steps we obtain a cutset $T = \{v_1, \dots, v_t\}$ that fulfils the hypothesis of Theorem \ref{Minheight}, and then $P_T(G) \in \minh(G)$. 
Now we explain in detail the algorithm. \\
Line 1. We set $\dim = n+1$. This is the case when the graph $G$ is a block or is decomposable into blocks, that is $T = \emptyset$.\\
Line 2. We denote by $\mathcal{G}$ the set of graphs that are to consider still. \\
Lines 3-4. We consider each graph $H \in \mathcal{G}$. The algorithm finishes when $\mathcal{G}$ is empty. \\
Lines 5-6. We decompose $H$ into its indecomposable components $G_1, \dots, G_r$. These subgraphs are the elements of the set $\mathcal{I}$. This is equivalent to do away with the vertices of clique degree 2 (see Remark \ref{T not vert of degree 2}). Now,  by a branch and bound strategy we study each indecomposable subgraphs of $H$.  We discard the blocks since their vertices are free vertices and then they do not belong to $T$.\\
Lines 7-8. For every subgraph $G_i \in \mathcal{I}$, since $G_i$ is indecomposable there exists a vertex $v$ that belongs to at least $2$ endblocks. By Lemma \ref{v with at least 2 leaves}, we assume $v \in T$. \\
Line 9. We update the Krull dimension: the number of connected components induced by $v$ in $G_i$ is exactly its clique degree. One of these components has been already considered, when we set $\dim = n+1$ in the Line 1. Therefore, the contribute of $v$ is equal to $\cdeg(v)-1$ less the cardinality of the cutset, which is 1.\\
Line 10. We remove from $G_i$ the vertex $v$ and the endblocks which contain $v$, and we add this new graph $H_v$ in $\mathcal{G}$, the set of graphs to consider still.

The wanted $T$ consists of all the vertices $v$ considered in Line 8. Observe that, by construction, any $v \in T$ satisfies the condition (1) of Theorem \ref{Minheight}, and the condition (2) holds at the end of the algorithm, when $\mathcal{G} = \emptyset$.
Moreover, the algorithm finishes after a finite number of steps: in Line 4, we remove a graph $H$ from $\mathcal{G}$ but we add some new graphs $H_v$ in $\mathcal{G}$ in Line 10. For any of these $H_v$, it holds $|V(H_v)| < |V(H)|$, hence after a finite number of iterations the new graphs in Line 10 will be either blocks, and then they will be discarded in Line 6, or empty graphs. 
\end{proof}

We highlight that the above algorithm works also for disconnected graphs: it is sufficient to set $\dim:=n+c$ in Line 1, where $c$ is the number of connected components of $G$. 

We are going to show that the Krull dimension of $S/J_G$ can be computed with a unique visit of $G$ by a recursive function, named $\IsInT$. The cost of traversing a graph $G$ is $\mathcal{O}(|V(G)|+|E(G)|)$ (see \cite[Section 22]{CLRS}). This implies that the Algorithm \ref{Alg for Krull dim} can be implemented through a procedure which is linear with respect to the number of vertices and edges of $G$, without any decomposition. The function $\IsInT$ constructs a $T \in \mathcal{C}(G)$ that fulfils the conditions (1) and (2) of Theorem \ref{Minheight}, and then $P_T(G) \in \minh(G)$. For the sake of simplicity, in the following let $G$ be a tree. We recall that a vertex $v$ is called a \textit{leaf} of a tree if $\cdeg(v) = 1$.

\begin{enumerate}[{\footnotesize \hspace{0.08cm} 1.}]
\item $\IsInT$ 
\item if $v$ is a leaf then 
\item \hspace{0.4cm} return 0 
\item else 
\item  \hspace{0.4cm} degree := $\cdeg(v)$; 
\item  \hspace{0.4cm} childrenInT := 0; 
\item  \hspace{0.4cm} for every children $w$ of $v$ 
\item  \hspace{0.8cm} childrenInT := childrenInT + $\IsInTw$; 
\item  \hspace{0.4cm} degree := degree - chidrenInT; 
\item  \hspace{0.4cm} if degree > 2 then 
\item  \hspace{0.8cm} $c_T$ :=  $c_T$ + degree - 2; 
\item  \hspace{0.8cm} return 1
\item  \hspace{0.4cm} else 
\item  \hspace{0.8cm} return 0
\end{enumerate}

Even if the algorithm works for any undirected tree, we assign an orientation given by the visit of the tree itself: the children of a given vertex are its adjacent vertices that have not been visited yet.
The purpose of $\IsInTv$ is twofold: on one side, starting from any vertex $v \in V(G)$, it checks if $v$ belongs to $T$ and in this case it returns 1, otherwise 0, on the other side it computes $c(T)-|T|$, which is saved in $c_T$.
For a vertex $v$ being in $T$ depends on its children that are in $T$, and on its degree. The latter is given by the initial degree less the number of children of $v$ that are in $T$ (Line 9). In particular, $v \in T$  if at least 2 of its children are not in $T$ and its degree is greater than 2 (Line 10).  

To compute the Krull dimension of $S/J_G$, it is sufficient to call the function $\IsInTv$, where $v$ is any vertex of $G$ and $c_T$ is a global variable set to 1, and then $\dim S/J_G = n+c_T$.


We have implemented this procedure for trees using CoCoA version 4.7 and it is freely downloadable  on \cite{MR}.

\section{Regularity bounds for binomial edge ideals of block graphs}
The main result of this section is the lower bound for the Castelnuovo-Mumford regularity of binomial edge ideals of block graphs (Theorem \ref{The:LowerBound}). To reach our result, we compute the regularity and the superextremal Betti numbers of special block graphs, called \textit{flower graphs}.

Let $M$ be a finitely generated graded $S$-module. A Betti number $\beta_{i,i+j}(M) \neq 0$ is called \textit{extremal} if $\beta_{k,k+\ell} =0$ for all pairs $(k,\ell) \neq (i,j)$, with $k \geq i, \ell \geq j$. 
Let $q=\reg M$ and $p=\mathrm{projdim} M$, then there exist unique numbers $i$ and $j$ such that
$\beta_{i,i+q}(M)$ and $\beta_{p,p+j}(M)$ are extremal Betti numbers. We call them the \textit{distinguished extremal Betti numbers} of $M$.
Let $k$ be the maximal integer $j$ such that $\beta_{i,j} \neq 0$ for some $i$. It is clear that $\beta_{i,k}(M)$ is an extremal Betti number for all i with $\beta_{i,k} \neq 0$, and that there is at least one such $i$. These Betti numbers are distinguished by the fact that
they are positioned on the diagonal $ \{ (i,k-1)| i=0,\dots,k \}$ in the Betti
diagram, and that all Betti numbers on the right lower side of the diagonal are zero. The Betti numbers $\beta_{i,k}$, for $i=0, \dots, k$, are called \textit{superextremal}, regardless of whether they are zero or not. We refer the reader to \cite[Chapter 11]{HH-book} for further details.\\

Let $G$ be a graph. We denote by $i(G)$ the number of inner vertices of $G$ and by $f(G)$ the number of free vertices of $G$. 

\begin{definition}\label{Def:BigFlower}
A flower graph $F_{h,k}(v)$ is a connected block graph constructed by joining $h$ copies of the cycle graph $C_3$ and $k$ copies of the bipartite graph $K_{1,3}$ with a common vertex $v$, where $v$ is one of the free vertices of $C_3$ and of $K_{1,3}$, and $\cdeg(v) \geq 3$.
\end{definition}

We observe that any flower graph $F_{h,k}(v)$ has $2h+3k+1$ vertices and $3(h+k)$ edges. The clique degree of $v$ is given by $h+k$, and the number of inner vertices is $i(F_{h,k}(v))=k+1$ and all of them are cutpoints for $F_{h,k}(v)$. When it is unnecessary to make explicit the parameters $h$ and $k$, we refer to $F_{h,k}(v)$ as $F(v)$.

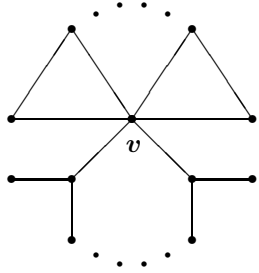
\begin{figure}[h]
\begin{center}
\setlength{\unitlength}{0.4cm}
\begin{picture}(8,9)
\newsavebox{\Tri}

\savebox{\Tri}
  (04,03)[bl]{
  \put(00,00){\circle*{.3}}
  \put(04,00){\circle*{.3}}
  \put(02,03){\circle*{.3}}

  \put(00,00){\line(2,3){2}}
  \put(00,00){\line(1,0){4}}
  \put(02,03){\line(2,-3){2}}
}

\put(00,03){\usebox{\Tri}}
\put(04,03){\usebox{\Tri}}

\put(04,03){\line(-1,-1){2}}
\put(02,01){\line(-1,0){2}}
\put(02,01){\line(0,-1){2}}

\put(04,03){\line(1,-1){2}}
\put(06,01){\line(1,0){2}}
\put(06,01){\line(0,-1){2}}

\put(02,01){\circle*{.3}}
\put(00,01){\circle*{.3}}
\put(02,-01){\circle*{.3}}
\put(06,01){\circle*{.3}}
\put(08,01){\circle*{.3}}
\put(06,-01){\circle*{.3}}

\put(02.8,6.5){\circle*{.2}}
\put(03.6,6.8){\circle*{.2}}
\put(04.4,6.8){\circle*{.2}}
\put(05.2,6.5){\circle*{.2}}

\put(02.8,-01.5){\circle*{.2}}
\put(03.6,-01.8){\circle*{.2}}
\put(04.4,-01.8){\circle*{.2}}
\put(05.2,-01.5){\circle*{.2}}

\put(03.8,1.9){\boldmath{$v$}}
\end{picture}

\end{center}
\vspace{0.5cm}
\caption{A flower graph $F_{h,k}(v)$}\label{Big Flower}
\end{figure}

\begin{remark}\label{decomp cactus}
Let $G$ be a flower graph $F(v)$. By the result \cite[Corollary 1.5]{R2}, $G = J_{G'} \cap Q_v$
where $G'$ is the graph obtained from $G$ by connecting all the vertices adjacent to $v$, and $Q_v = \bigcap_{ T \in \mathcal{C}(G), v \in T} P_T(G)$. We observe that in this case $Q_v = (x_v,y_v) + J_{G''}$, where $G''$ is obtained from $G$ by removing $v$, and then 
\[
J_G = J_{G'} \cap ((x_v,y_v) + J_{G''}).
\]
\end{remark}

Before stating the distinguished extremal Betti numbers of the binomial edge ideal of a flower graph, we need the following remark. 

\begin{remark}\label{betti numb for disconnected graph}
Let $G$ be a disconnected block graph with $G_1, \dots, G_r$ its connected components. If all the $G_j$ have precisely one extremal Betti number, $\beta_{n_j-1,n_j+i(G_j)}(S_j/J_{G_j})$, for any $j=1,\dots, r$, with $S_j= K[x_i,y_i]_{i \in V(G_j)}$ and $n_j = |V(G_j)|$, then $S/J_G$ has precisely one extremal Betti number and it is given by
$$ \beta_{n-r,n+i(G)}(S/J_G)= \prod_{j=1}^{r} \beta_{n_j-1,n_j+i(G_j)}(S_j/J_{G_j}).$$

\end{remark}

\begin{theorem}\label{main theo}
 Let $G$ be a flower graph $F(v)$. The following are extremal Betti numbers of $S/J_G$:
 \begin{enumerate}
   \item $\beta_{n-1,n+i(G)}(S/J_G)=f(G)-1;$
   \item $\beta_{n-\cdeg(v)+1,n+i(G)}(S/J_G)=1.$
\end{enumerate}
In particular, they are the only non-zero superextremal Betti numbers.
\end{theorem}
\begin{proof}
The fact (1) is proved in \cite[Theorem 2.2]{HR}. As regards (2), we focus on the cutpoint $v$ of $G$. Thanks to the decomposition quoted in Remark \ref{decomp cactus}, we consider the following exact sequence
\begin{equation}\label{Exact}
 0\To S/J_G \To S/J_{G'}\oplus S/((x_v, y_v)+J_{G''})\To S/((x_v,y_v)+J_{H}) \To 0 
\end{equation}
where $G'$ and $G''$ are described in Remark \ref{decomp cactus}, and $H$ is obtained from $G'$ by removing $v$.
We observe that $G'$ and $H$ are block graphs satisfying \cite[Theorem 2.4 (b)]{HR}, with $i(G')=i(H)=i(G)-1$, and then $\reg S/J_{G'} = \reg S/((x_v,y_v)+J_H) = i(G)$. The graph $G''$ has $\cdeg(v)$ connected components $G_1,\ldots, G_{\cdeg(v)}$: all of them are either $K_2$ or paths of length 2, namely $P_2$. The latter are decomposable into two $K_2$ and it holds $\reg S'/J_{P_2} = 2 = i(P_2)+1$, with $S'=K[x_i,y_i]_{i \in V(P_2)}$. Then, by \cite[Theorem 2.4 (b)]{HR} and since the ring $S/((x_v,y_v)+J_{G''})$ is the tensor product of $S_j/J_{G_j}$, with $j=1,\dots, \cdeg(v)$ and $S_j = K[x_i,y_i]_{i \in V(G_j)}$, we have
\[
 \reg \frac{S}{(x_v,y_v)+J_{G''}} =\sum_{j=1}^{\cdeg(v)} \reg \frac{S_j}{J_{G_j}}=\sum_{j=1}^{\cdeg(v)}  (i(G_j)+1)=i(G)-1+\cdeg(v).
\]
We get the following bound on the regularity of $S/J_G$

\begin{eqnarray*}
  \reg S/J_G\hspace{-0.2cm}&\leq &\hspace{-0.2cm}\max\{\reg \frac{S}{J_{G'}},\reg \frac{S}{(x_v, y_v)+J_{G''}}, \reg \frac{S}{(x_v,y_v)+J_{H}} +1\}\\
             &= &\hspace{-0.2cm}\max\{i(G), i(G)-1+\cdeg(v), i(G)+1\}\\
             &= &i(G)-1+\cdeg(v).
\end{eqnarray*}

By \cite[Theorem 1.1]{EHH}, the depth of $S/J_G$ for any block graph $G$ over $[n]$ is equal to $n+c$, where $c$ is the number of connected components of $G$. Since we know the depth of all quotient rings involved in (\ref{Exact}) and by Auslander-Buchsbaum formula, we get $\projdim S/J_G=\projdim S/J_{G'}=\projdim S/((x_v,y_v)+J_H)-1=n-1$, 
and $\projdim S/((x_v,y_v)+J_{G''})=n-\cdeg(v)+1$.

Let $j> i(G)$, then 
\[
T_{m,m+j}(S/J_{G'})=T_{m,m+j}(S/((x_v,y_v)+J_H)) =0 \qquad  \text{for any } m,
\]
and 
\[
T_{m,m+j}(S/((x_v,y_v)+J_{G''})) =0 \qquad \text{for any } m > n-\cdeg(v)+1,
\] 
where $T_{m,m+j}^S(M)$ stands for $\Tor_{m,m+j}^S(M,K)$ for any $S$-module $M$, and $S$ is omitted if it is clear from the context. 
Of course, all the above Tor modules $T_{m,m+j}(-)$ are zero when $j > i(G)-1+\cdeg(v)$. 

Therefore, for $m=n-\cdeg(v)+1$ and $j=i(G)-1+\cdeg(v)$ we obtain the following long exact sequence
\begin{eqnarray*}\label{longexact}
&\cdots&\rightarrow T_{m+1,m+1+(j-1)}(S/((x_v,y_v)+J_H)) \rightarrow T_{m,m+j}(S/J_G) \rightarrow  \\
&\ & T_{m,m+j}(S/J_{G'}) \oplus T_{m,m+j}(S/((x_v, y_v)+J_{G''})) \rightarrow \\
&\ & T_{m,m+j}(S/((x_v,y_v)+J_H)) \rightarrow \cdots
\end{eqnarray*}
In view of the above, all the functors on the left of $T_{m,m+j}(S/J_G)$ in the long exact sequence are zero, and  $T_{m,m+j}(S/J_{G'})=T_{m,m+j}(S/((x_v,y_v)+J_H))=0$ too. It follows 
\begin{equation*}
T_{m,m+j}(S/J_G) \cong T_{m,m+j}(S/((x_v, y_v)+J_{G''})).
\end{equation*}
It means that 
\[
\beta_{n-\cdeg(v)+1,n+i(G)}(S/J_G) = \beta_{n-\cdeg(v)+1,n+i(G)}(S/((x_v, y_v)+J_{G''})).
\] 
We observe that 
\[
T^S_{m,m+j}(S/((x_v, y_v)+J_{G''})) \cong T^{S''}_{m-2,m-2+j}(S''/J_{G''})
\]
where $S''=S/(x_v,y_v)$. Since all the connected components $G_1,\ldots, G_{\cdeg(v)}$ of $G''$ are either a $K_2$ or a path of length 2, the quotient rings $S_j/J_{G_j}$ have an unique extremal Betti number $\beta_{n_j-1,n_j+i(G_j)}(S_j/J_{G_j})$, for $j=1, \dots, \cdeg(v)$ and $n_j=|V(G_j)|$, which is equal to 1. Therefore, by Remark \ref{betti numb for disconnected graph}, we have 
$$\beta_{m-2,m-2+j}(S''/J_{G''}) = \prod_{j=1}^{\cdeg(v)} \beta_{n_j-1,n_j+i(G_j)}(S_j/J_{G_j}) =1.$$
Observe that for $m=n-\cdeg(v)+1$ and $j=i(G)-1+\cdeg(v)$ we get that $m+j=n+i(G)$ is the maximal integer such that $\beta_{i,m+j}(S/J_G) \neq 0$ for some $i$. We want to prove that $\beta_{i,n+i(G)} \neq 0$, only for $i=n-\cdeg(v)+1$ and $i=n-1$. Let $i$ be an integer such that $\beta_{i,n+i(G)} \neq 0$. Since $\projdim S/J_G =n-1$ and $\reg S/J_G \leq i(G) + \cdeg(v) -1$, we have to examine $ n-\cdeg(v)+1 \leq i \leq n-1$. Consider the following long exact sequence
\begin{eqnarray*}\label{longexact}
  &\cdots& \rightarrow T_{i+1,n+i(G)}\left(\frac{S}{(x_v,y_v)+J_H}\right) \rightarrow T_{i,n+i(G)}\left(\frac{S}{J_G}\right) \rightarrow  \\
  &\ & T_{i,n+i(G)}\left(\frac{S}{J_{G'}}\right) \oplus T_{i,n+i(G)}\left(\frac{S}{(x_v, y_v)+J_{G''}}\right) \rightarrow \\
 &\ & T_{i,n+i(G)}\left(\frac{S}{(x_v,y_v)+J_H}\right) \rightarrow   \cdots
\end{eqnarray*}
If $ n-\cdeg(v)+1 < i < n-1$, since $i > \projdim S/((x_v,y_v)+J_{G''})$ and  $n+i(G)-i > \reg S/J_{G'}, \reg S/((x_v,y_v)+J_H)$, it holds $\Tor_{i,n+i(G)}(M) =0$, for  $M\in \{S/J_{G'}, S/((x_v,y_v)+J_{G''}),  S/((x_v,y_v)+J_H)\}$,
and then we can conclude that also $\Tor_{i,n+i(G)}(S/J_G)=0$.
\end{proof}

An immediate consequence of the proof of the Theorem \ref{main theo} is the regularity of any flower graphs $F(v)$, that depends only on the clique degree of $v$ and the number of inner vertices of $F(v)$. 

\begin{corollary}\label{cor:reg flower}
Let $F(v)$ be a flower graph, then 
\[
\reg S/J_{F(v)}=i(F(v)) + \cdeg(v)-1.
\]
\end{corollary}

If $F(v)$ is an induced subgraph of a block graph $G$, we denote by $\cdeg_F(v)$ the clique degree of $v$ in $F(v)$. Note that if $F(v)$ is the maximal flower induced subgraph of $G$ and all the blocks of $G$ containing $v$ are $C_3$ or $K_{1,3}$, then $\cdeg_F(v) = \cdeg(v)$, otherwise $\cdeg_F(v) < \cdeg(v)$.

\begin{theorem}\label{The:LowerBound}
Let $G$ be an indecomposable block graph and let $F(v)$ be an induced subgraph of $G$. Then
 \[
  \reg S/J_G \geq i(G)+\cdeg_F(v)-1.
 \]
\end{theorem}
\begin{proof}
We use induction on the number of blocks of $G$ that are not in $F(v)$. If $G=F(v)$, the statement follows from Corollary \ref{cor:reg flower}. Suppose now $G$ contains properly $F(v)$ as induced subgraph. Since $G$ is connected, there exists an endblock $B$ of $G$ and a subgraph $G'$ of $G$ such that $G=G' \cup B$, $G'$ contains  $F(v)$ as induced subgraph, $V(G')\cap V(B)=\{w\}$, and all the blocks containing $w$ are endblocks, except for the one that is in $G'$. Since $G$ is assumed to be indecomposable, $\cdeg(w) \geq 3$. If $\cdeg(w) = 3$, then $G'$ is decomposable into $G_1 \cup G_2$, and $\reg S/J_{G'} = \reg S/J_{G_1} +\reg S/J_{G_2}$. We may suppose that $G_1$ contains $F(v)$, and then $i(G_1) = i(G)-1$, but $\cdeg_F(v)$ is still the same. Whereas, $G_2$ is a block and $\reg S/J_{G_2} = 1$. Then by using induction, we may assume that $\reg S/J_{G_1} \geq i(G)+\cdeg_F(v)-2$. Therefore, 
\[
\reg S/J_{G'}= \reg S/J_{G_1}+\reg S/J_{G_2} \geq i(G)+\cdeg_F(v)-1.
\]
If $\cdeg(w) > 3$, then $i(G')=i(G)$ and $\cdeg_F(v)$ is still the same. Then, by using induction on the number of blocks of $G$, we may assume $\reg S/J_{G'} \geq i(G)+\cdeg_F(v)-1$.
By \cite[Corollay 2.2]{MM} of Matsuda and Murai, one have that 
\[
\reg S/J_{G} \geq \reg S/J_{G'}.
\]
and then $\reg S/J_{G} \geq i(G)+\cdeg_F(v)-1$, as desired.
\end{proof}

\begin{definition} 
Let $G$ be a block graph. If $G$ has no flower graphs as induced subgraphs then $G$ is called {\em flower-free}.
\end{definition}

We are ready to state the following bound for the regularity for any binomial edge ideal of block graphs.

\begin{corollary}\label{Cor:LowerBound}
Let $G$ be a connected block graph which is not an isolated vertex. 
\begin{enumerate}
\item If $G$ is a flower-free graph, then $\reg S/J_G= i(G)+1$.
\item If $G$ contains $r \geq 1$ flower graphs $F_1(v_1),\dots,F_r(v_r)$ as induced subgraphs, then $\displaystyle \reg S/J_G \geq i(G)+\max_{i=1,\dots,r}\{\cdeg_{F_i}(v_i)\}-1$.
\end{enumerate}

\end{corollary}

\begin{proof}
(1) If $G$ is indecomposable, by \cite[Theorem 2.4]{HR}, the result follows. Otherwise, suppose $G$ is decomposable into indecomposable graphs $G_1, \dots, G_r$. Observe that if $v$ is an inner vertex in $G$ then either $\{v\}=G_i \cap G_j$ for some $i \neq j$ and it is a free vertex in $G_i$ and $G_j$, or it belongs to an unique $G_i$ and it is an inner vertex of $G_i$. The former are exactly $r-1$. In fact, if we consider the graph $T$, with vertices $V(T)=\{G_1,\ldots, G_r\}$ and edges $E(T)=\{\{G_i,G_j\}: G_i\cap G_j\neq \emptyset\}$ we observe that $T$ is a tree and $|E(T)|=r-1$. Hence 
\[
i(G) = r -1 + \sum_{i=1}^r i(G_i).
\]
By Proposition \ref{prop:sum of invar} and \cite[Theorem 2.4]{HR}, we get
\[
\reg S/J_G = \sum_{i=1}^r \reg S/J_{G_i} = \sum_{i=1}^r (i(G_i) + 1) = i(G)+1.
\]

\noindent (2) It is an immediate consequence of \cite[Corollary 2.2]{MM} and Theorem \ref{The:LowerBound}.

\end{proof}

\begin{example}\label{ex:2flowers}
Let $G$ be the graph in Figure \ref{Fig:FinalExample}. It contains 2 flower graphs as induced subgraphs: $F_{2,1}(v_1)$ and $F_{3,1}(v_2)$. By Corollary \ref{Cor:LowerBound}, we have $\reg S/J_G \geq 2 + \max \{3,4\} - 1 = 5$, whereas the length of the longest induced path in $G$ is 3 and the number of maximal cliques of $G$ is 6. Also using the upper bound proved in \cite{JNR2}, we get $\reg S/J_G \leq 6$. By means of a computation in CoCoA, $\reg S/J_G  =5$, it means the lower bound given in Corollary \ref{Cor:LowerBound} is sharp. We observe that $G$ is the graph with the minimum number of vertices such that $S/J_G$ has 3 non-zero superextremal Betti numbers. 

\end{example}

\begin{figure}[h]
\begin{center}
\setlength{\unitlength}{0.4cm}
\begin{picture}(12,3)
\newsavebox{\TriCapovolto}
\newsavebox{\TriSdraiato}

\savebox{\Tri}
  (04,03)[bl]{
  \put(00,00){\circle*{.3}}
  \put(04,00){\circle*{.3}}
  \put(02,03){\circle*{.3}}

  \put(00,00){\line(2,3){2}}
  \put(00,00){\line(1,0){4}}
  \put(02,03){\line(2,-3){2}}
}

\savebox{\TriCapovolto}
  (04,03)[bl]{
  \put(02,00){\circle*{.3}}
  \put(00,03){\circle*{.3}}
  \put(04,03){\circle*{.3}}

  \put(02,00){\line(2,3){2}}
  \put(02,00){\line(-2,3){2}}
  \put(00,03){\line(4,0){4}}
}

\savebox{\TriSdraiato}
  (04,03)[bl]{
  \put(00,00){\circle*{.3}}
  \put(03,02){\circle*{.3}}
  \put(03,-02){\circle*{.3}}

  \put(00,00){\line(3,2){3}}
  \put(00,00){\line(3,-2){3}}
  \put(03,02){\line(0,-4){4}}
}

\put(00,00){\usebox{\TriCapovolto}}
\put(00,-03){\usebox{\Tri}}

\put(02,00){\line(8,0){8}}

\put(8,00){\usebox{\TriCapovolto}}
\put(8,-03){\usebox{\Tri}}
\put(10,-02){\usebox{\TriSdraiato}}

\put(02.5,0.2){\boldmath{$v_1$}}
\put(8.5,0.2){\boldmath{$v_2$}}

\end{picture}
\end{center}

\vspace{1.5cm}

\caption{A graph $G$ such that $\displaystyle \reg S/J_G = i(G)+\max_{i=1,2}\{\cdeg_{F_i}(v_i)\}-1$.}\label{Fig:FinalExample}

\end{figure}

Example \ref{ex:2flowers} encourages us to follow up with an algorithm to compute the regularity of binomial edge ideal of block graphs, and it will be the content of the section \ref{sec:algreg}.

The bound exhibited in Corollary \ref{Cor:LowerBound} can be improved for block graphs with several flowers $F_i(v_i)$ with the vertices $v_i$ far enough from each other. In particular, let $H$ be an induced subgraph of $G$ and suppose $H$ is decomposable into $H_1,\ldots, H_r$ such that any $H_i$ contains a flower graph $F_i(v_i)$ as an induced subgraph for $i=1,\ldots,r$. Then
$\reg J_G\geq \reg J_{H}=\sum_{i=1}^r\reg J_{H_i}$, which could be better than the one provided in Corollary \ref{Cor:LowerBound}.

\section{How to compute the Castelnuovo-Mumford regularity of block graphs}\label{sec:algreg}
In this section we provide an efficient method to compute the Castelnuovo-Mumford regularity for $S/J_G$ when $G$ is a block graph.

\begin{definition}\label{Def: end-flowers}
Let $G$ be a block graph and $F(v)$ be a flower graph that is an induced subgraph of $G$. $F(v)$ is called an {\em end-flower} of $G$ if $G=G_1\cup \ldots \cup G_c$, where $c=\cdeg(v)$, and such that $G_i\cap G_j=\{v\}$, for all $1 \leq i  < j \leq c$, and $G_2,\ldots G_c$ are flower-free graphs.
\end{definition}

\begin{theorem}\label{Theo: reg block graph}
Let $G$ be a block graph, $v_1,\ldots, v_r\in V(G)$, $$H_j=G\setminus\{v_1,\ldots,v_{j}\}$$ for $j=1,\ldots,r$, and $H_0=G$. If
\begin{enumerate}
 \item $F(v_j)$ is an end-flower for $H_{j-1}$, for all $j=1,\ldots,r$,
 \item $H_r$ is flower-free,
\end{enumerate}
then
\[
 \reg S/J_G=\reg S/J_{H_r}=c+i(H_r) 
\]
where $c$ is the number of connected components of $H_r$ which are not isolated vertices.
\end{theorem}
\begin{proof}
First of all, observe that the equality $\reg S/J_{H_r} = c+i(H_r)$ in the statement is an immediate consequence of Corollary \ref{Cor:LowerBound} (1).

 To prove  $\reg S/J_G=\reg S/J_{H_r}$, we make induction on 
 \[
  f= |\{ v \in V(G) | F(v) \text{ is an induced subgraph of } G\}|.
 \]
If $f=0$, that is $G$ is a flower-free graph, then the assertion follows by Corollary \ref{Cor:LowerBound} (1). Let $f=1$ and $v$ be such that $F(v)$ is an induced subgraph of $G$. Consider the exact sequence
\begin{equation}\label{Exact seq}
 0\To S/J_G \To S/J_{G'}\oplus S/((x_v, y_v)+J_{G''})\To S/((x_v,y_v)+J_{H}) \To 0
\end{equation}
where $G'$, $G''$, and $H$ are as described in Remark \ref{decomp cactus}. We observe that $G'$, $G''$, and $H$ are flower-free. Hence
\[
 \reg S/J_{G'}=\reg S/J_H=i(G')+1=i(G)-1+1=i(G).
\]
Moreover, removing the vertex $v$ from $G$ we obtain $G''$ and $\reg S/J_{G''}$ is
\[
 \sum_{j=1}^c \reg S/J_{G_j}=\sum_{j=1}^c (i(G_j)+1)=\sum_{j=1}^{\cdeg_{F}(v)} (i(G_j)+1)+ \sum_{k=1}^{c'}(i(G_k)+1)
 \]
where $G_1, \dots, G_c$ are the connected components of $G''$, and $\{v,w_k\}$ are maximal cliques in $G$ with $w_k$ a free vertex of $G_k$, and $|V(G_k)| \geq 2$, for $k=1,\ldots,c'$. Observe that, for $j=1,\ldots,\cdeg_{F}(v)$, all the inner vertices of $G$ that belong to $G_j$ are inner vertices also in $G''$. Whereas, for $k=1,\ldots,c'$, the $w_k$ are inner vertices in $G$ but not in $G''$, and all the other inner vertices of $G$ that belong to $G_k$ are inner vertices also in $G''$. Hence, removing $v$ from $G$, we have $c'+1$ less inner vertices in $G''$ with respect to $G$, that are all the $w_k$ and $v$, but this is compensated by the formula $\sum_{k=1}^{c'} (i(G_k) +1) = c' + \sum_{k=1}^{c'} i(G_k) $. Hence
\[
 \reg S/J_{G''}=i(G)+\cdeg_{F}(v)-1.
\]
Since $\cdeg_F(v)\geq 3$, 
\[
 \reg S/J_{G'}, \reg S/((x_v,y_v)+J_{H})  < \reg S/((x_v, y_v)+J_{G''})
\]
and then $\reg S/J_G = \reg S/((x_v, y_v)+J_{G''})$.

Let $f > 1$. Let $v_1, \dots, v_r \in V(G)$ be a sequence that fulfills (1) and (2). Consider the exact sequence (\ref{Exact seq}), with $v=v_1$. Observe that the sequence $v_2, \dots, v_r$ satisfies (1) and (2) for $G'$, $G''$, and $H$ and, since they have less than $f$ flower graphs as induced subgraphs, by induction hypothesis their regularity is given by the sum of the regularity of the connected components induced by $v_2, \dots, v_r$. 

Let $G_1, \dots, G_m$ be the connected components induced by $v_2, \dots, v_r$ in $G$. One of them contains $v_1$, suppose $G_1$, and then it is not flower-free, whereas the others are flower-free. The connected components induced by $v_2, \dots, v_r$ in $G'$ and $H$ are $G_1', G_2, \dots, G_m$ and $G_1'\setminus \{v_1\}, G_2, \dots, G_m$, respectively, where $G_1'$ denotes the graph obtained from $G_1$ by connecting all the vertices adjacent to $v_1$. We get 
\[
\reg S/J_{G'} = \reg S/J_H = \reg S/J_{G_1'} +  \sum_{i=2}^m \reg S/J_{G_i}.
\]
Whereas, the connected components induced by $v_2, \dots, v_r$ in $G''$ are the connected components of $G_1\setminus \{v_1\}$ and $G_2, \dots, G_m$, and then
\[
\reg S/J_{G''} = \reg S/J_{G_1\setminus \{v_1\}} +  \sum_{i=2}^m \reg S/J_{G_i}.
\]
Since
\[
\reg S/J_{G_1'} = i(G_1) < i(G_1) + \cdeg_{F}(v_1) -1 =  \reg S/J_{G_1\setminus \{v_1\}},
\]
where the last equality follows from the same arguments of above and $\cdeg_{F}(v_1)$ denotes the clique degree of $v_1$ in $F(v_1)$, with $F(v_1)$ seen as induced subgraph of $G_1$. Since $F(v_1)$ is an end-flower and $\cdeg_{F}(v_1) \geq 3$ in $G$, it follows $\cdeg_{F}(v_1)\geq 2$ in $G_1$. Observe that, when $\cdeg_{F}(v_1)= 2$ in $G_1$, $G_1$ is flower-free and it is easy to see that the equality $\reg S/J_{G_1\setminus \{v_1\}} = i(G_1) + \cdeg_{F}(v_1) -1$ is still true. Then 
\[
 \reg S/J_{G'}, \reg S/((x_v,y_v)+J_{H})  < \reg S/((x_v, y_v)+J_{G''})
\]
and the assertion is proved. 
\end{proof}

The Theorem \ref{Theo: reg block graph} suggests a recursive way to compute the regularity of $S/J_G$ when $G$ is a block graph. 

\begin{enumerate}[{\footnotesize \hspace{0.08cm} 1.}]
\item ComputeRegularity($G$) 
\item if $G$ is flower-free and is not an isolated vertex 
\item \hspace{0.4cm} return $i(G)+1$ 
\item else 
\item  \hspace{0.4cm} reg := 0; 
\item  \hspace{0.4cm} pinpoint an end-flower $F(v)$ of $G$; 
\item  \hspace{0.4cm} remove $v$ from $G$;  
\item  \hspace{0.4cm} for every connected component $G_i$ induced by $v$ in $G$ 
\item  \hspace{0.8cm} reg := reg + ComputeRegularity($G_i$); 
\item  \hspace{0.4cm} return reg 
\end{enumerate}

\ 

By means of an unique block graph traversal, that is linear with respect to the number of vertices and edges of $G$ (see \cite[Section 22]{CLRS}), one get the regularity of $S/J_G$. This allows to compute the regularity of $S/J_G$ also for those block graphs with a large number of vertices, and then for those binomial edge ideals with a large number of variables, for which the algebraic softwares, as CoCoA \cite{Co} and Macaulay2 \cite{Ma}, fail.

\end{document}